\documentclass[11pt]{amsart}
 \font\tenmsb=msbm10 scaled\magstep
1
\newfam\msbfam
   \textfont\msbfam=\tenmsb

\usepackage{amsmath}
\usepackage{amssymb}
\usepackage{amsrefs}
\usepackage{mathscinet}
\usepackage{amsfonts,mathrsfs, bbding}
\usepackage{epsfig}
\usepackage{amsthm}
\usepackage{bm}
\usepackage{latexsym}
\usepackage[latin1]{inputenc}
\usepackage{verbatim}
\usepackage{pb-diagram}
\usepackage[all]{xy}

\newtheorem{theorem}{Theorem}[section]

\newtheorem{lemma}[theorem]{Lemma}
\newtheorem{corollary}[theorem]{Corollary}
\newtheorem{proposition}[theorem]{Proposition}
\theoremstyle{definition}
\newtheorem{definition}[theorem]{Definition}
\newtheorem{example}[theorem]{Example}

\theoremstyle{remark}
\newtheorem{remark}[theorem]{Remark}

\newtheorem{observation}[theorem]{Observation}

\numberwithin{equation}{section}

\newcommand{\m}{\mathcal }

\def\limind{\mathop{\oalign{lim\cr
\hidewidth$\longrightarrow$\hidewidth\cr}}}
\DeclareMathAlphabet{\mathpzc}{OT1}{pzc}{m}{it}

\newcommand{\compcent}[1]{\vcenter{\hbox{$#1\circ$}}}

\newcommand{\comp}{\mathbin{\mathchoice
{\compcent\scriptstyle}{\compcent\scriptstyle}
{\compcent\scriptscriptstyle}{\compcent\scriptscriptstyle}}} 

\begin{document}

\title{Product systems arising from L\'evy processes}

\author{Remus Floricel}
\author{Peter Wadel}
\address{University of Regina, Department of Mathematics and Statistics, Regina, SK, Canada}
\email{Remus.Floricel@uregina.ca} 
\address{University of Waterloo, Faculty of Mathematics, Waterloo, ON, Canada}
\email{peter.wadel@uwaterloo.ca} 
\thanks{The first author's research was partially funded by a Discovery Grant from NSERC, while the second author received support through an NSERC USRA}

\begin{abstract}
This paper investigates the structure of product systems of Hilbert spaces derived from Banach space-valued L\'evy processes. We establish conditions under which these product systems are completely spatial and show that Gaussian L\'evy processes with non-degenerate covariance always give rise to product systems of type I. Furthermore, we construct a continuum of non-isomorphic product systems of type \(\rm{II}\sb\infty\) from pure jump L\'evy processes. 
\end{abstract}

\maketitle
\section{Introduction}

Introduced by Arveson in his seminal work \cite{A2}, product systems of Hilbert spaces have become fundamental in the classification of \(E_0\)-semigroups \cite{Pow88}, with significant connections to various areas of operator algebras and both classical and quantum probability (see \cite{A7} for a comprehensive overview). From the early days of the theory, strong connections between product systems and stochastic processes, particularly L\'evy processes, have been observed \cite{Mey, Mar0, Tsirelson2002}, and these links continue to be explored today in various settings (see, for example, \cite{Tsirelson2004, Liebscher, Skeide, Sundar} and references therein).

In \cite{Mar0, Mar1}, Markiewicz showed that product systems arising from \(\mathbb{R}^n\)-valued L\'evy processes are completely spatial (type I). However, it has long been proposed (see, for example, \cite{Skeide}) that product systems arising from L\'evy processes with state spaces beyond \(\mathbb{R}^n\) could be of type II. We provide a positive answer to this within the framework of Banach space-valued L\'evy processes.

The paper is organized as follows: Section \ref{sec:background} provides the necessary background on L\'evy processes and product systems of Hilbert spaces. In Section \ref{sec:product_system}, we construct the product system associated with a Banach space-valued L\'evy process and prove key properties, including general conditions ensuring the complete spatiality with respect to exponential units of these product systems (see Theorem \ref{Th1}). We also consider a Hilbert space-valued compound Poisson process in Example \ref{exm1} to construct a non-exponential unit in the associated product system.

In Section \ref{sec:pureGauss}, for purely Gaussian L\'evy processes with non-degenerate covariance, we prove that every unit in the associated product system is equivalent to a standard exponential unit derived from the Cameron-Martin space of the process (Theorem \ref{main1}). Note that any such unit is a martingale (Proposition \ref{Lem1}). In particular, we establish that the product systems associated with such Gaussian L\'evy processes are completely spatial (Theorem \ref{theorem:completely_spatial}).

In Section  \ref{sec:examples}, we construct a family  \(\{\bm{X}^\lambda\}_{\lambda\in  [0,1]^\mathbb{N}}\) of Skellam-type infinite-dimensional compound Poisson processes and show that their product systems are of type \(\mathrm{II}\sb{\infty}\) (Theorem \ref{type}) and mutually non-isomorphic (Theorem \ref{last}). This is obtained through a detailed description and analysis of the units of the product systems.

We note at the end that, with very few exceptions, all the notations and conventions used in this article are standard.

\section{Background}\label{sec:background}

This section provides the necessary background on both Banach space-valued L\'evy processes and product systems of Hilbert spaces, which are relevant to this work. The standard references for these topics are \cite{Sato1, Applebaum} for L\'evy processes and \cite{A7, A2} for product systems.

\subsection{L\'evy Processes on Banach Spaces}

Let \((B, \|\cdot\|)\) be a real, separable Banach space. A \(B\)-valued L\'evy process is a stochastic process \(\bm{L} = \{L_t\}_{t \geq 0}\) defined on a probability space \((\Omega, \mathcal{F}, \mathbb{P})\) with state space \(B\), satisfying the following conditions:
\begin{enumerate}
    \item  \textit{Stationary Increments:} For any \(s,\,t \geq 0\), the increment \(L_{t,t+s} := L_{t+s} - L_t\) is independent of \(L_u\) for \(u \leq t\) and has the same distribution as \(L_s\).
    \item  \textit{Independent Increments:} The increments \(L_{t_1, t_2} = L_{t_2} - L_{t_1}\), \(L_{t_2, t_3} = L_{t_3} - L_{t_2}\), \ldots, \(L_{t_{n-1}, t_n} = L_{t_n} - L_{t_{n-1}}\) are independent for any choice of \(t_1 \leq t_2 \leq \cdots \leq t_n\).
    \item \textit{Stochastic Continuity:} For every \(\varepsilon > 0\) and \(s \geq 0\), we have
    \[
    \lim_{t \downarrow 0} \mathbb{P}(\|L_{s, s+t}\| > \varepsilon) = 0.
    \]
    \item \textit{C\`adl\`ag Paths:} The paths \(t \mapsto L_t(\omega)\) are almost surely c\`adl\`ag, i.e., right-continuous with left limits with respect to the norm.
    \item \textit{Initial Condition:} \(L_0 = 0\) almost surely.
\end{enumerate}

We assume without loss of generality that \(\mathcal{F} = \sigma\{L_t \mid t \geq 0\}\). The natural filtration \(\{\mathcal{F}_t^0\}_{t \geq 0}\) of \(\bm{L}\) is given by \(\mathcal{F}_t^0 = \sigma\{L_s \mid s \leq t\}\). Setting \(\mathcal{F}_t = \mathcal{F}_t^0 \vee \mathcal{N}\), where \(\mathcal{N}\) denotes the \(\mathbb{P}\)-null sets of \(\mathcal{F}\), the filtered probability space \((\Omega, \mathcal{F}, \{\mathcal{F}_t\}_{t \geq 0}, \mathbb{P})\) satisfies the usual hypotheses, and the process \(\bm{L}\) is adapted to the filtration \(\{\mathcal{F}_t\}_{t \geq 0}\).

In the context of L\'evy processes, it is standard to model the sample space \(\Omega\) as the space \(D([0, \infty), B)\) of c\`adl\`ag functions from \([0,\infty)\) to \(B\), and \(\mathcal{F}\) as the Borel \(\sigma\)-algebra on \(D([0, \infty), B)\) generated by the open sets in the Skorokhod topology. In this setting, \(\bm{L} = \{L_t\}_{t \geq 0}\) is realized as the coordinate process on \(\Omega\), \(L_t(\omega) = \omega(t)\), and \(\mathbb{P}\) is the unique probability measure on \(\Omega\) that extends the finite-dimensional distributions of \(\bm{L}\), obtained using Kolmogorov's extension theorem. While the standard model will not be used by default in this article, it will be applied implicitly on several occasions, as will be clear from the context.

The characteristic function of any \(B\)-valued L\'evy process \(\bm{L} = \{L_t\}_{t \geq 0}\) is given by:
\[
\mathbb{E}\left[e^{i\langle \phi, L_t \rangle}\right] = \exp(t \Psi_{\bm{L}}(\phi)),
\]
for all \(\phi\) in the dual space \(B^*\) of \(B\). The L\'evy exponent of \(\bm{L}\), \(\Psi_{\bm{L}}: B^* \to \mathbb{C}\), is given by the L\'evy-Khintchine formula:
\[
\Psi_{\bm{L}}(\phi) = i \langle \phi, b \rangle - \frac{1}{2} Q(\phi, \phi) + \int_{B \setminus \{0\}} \left(e^{i \langle \phi, x \rangle} - 1 - i \langle \phi, x \rangle \mathbf{1}_{\|x\| \leq 1}\right) \nu(dx),
\]
for all \(\phi \in B^*\), where 
\begin{enumerate}
    \item[(i)] \(b \in B\) is a drift vector;
    \item[(ii)] \(Q: B^* \times B^* \to \mathbb{R}\) is the covariance of \(\bm{L}\), representing the Gaussian component of the process; and
    \item[(iii)] \(\nu\) is a L\'evy measure on \(B\), which satisfies:
    \[
    \int_{B} \min(1, \|x\|^2) \, \nu(dx) < \infty.
    \]
    The L\'evy measure \(\nu\) describes the jump behavior of the process.
\end{enumerate}
The triplet \((b, Q, \nu)\), referred to as the L\'evy-Khintchine triplet of \(\bm{L}\), uniquely determines the law of \(\bm{L}\).

The L\'evy-It\^o decomposition states that any \(B\)-valued L\'evy process \(\bm{L} = \{L_t\}_{t \geq 0}\) can be split into three independent components, given by the expression
\[
L_t = bt + W_t + J_t, \quad t \geq 0,
\]
where \(bt\) is the deterministic drift, \(W_t\) represents the Wiener process (which has continuous paths and corresponds to the Gaussian part), and \(J_t\) is the jump process, determined by the L\'evy measure \(\nu\), capturing the discontinuous jumps.

\subsection{Product Systems of Hilbert Spaces}

A product system is a bundle 
\[
E = \{(t, x) \mid x \in E(t),\, t > 0\}
\]
of infinite-dimensional, separable Hilbert spaces \((E(t), \langle \cdot, \cdot \rangle_{E(t)})\), which also carries the structure of a standard Borel space. The Borel structure must be compatible with the bundle structure, the vector space operations, and the inner products. The bundle also includes a ``multiplication'' represented by a measurable family \(\{U_{s, t}\}_{s,\,t > 0}\) of unitary operators
\[
E(s) \otimes E(t) \ni x \otimes y \longmapsto U_{s,t}(x \otimes y) =: xy \in E(t+s), \quad s,\, t > 0,
\]
which satisfy the associativity condition:
\begin{equation}\label{intro2}
U_{t_1, t_2+t_3}(\mathbf{1}_{E(t_1)} \otimes U_{t_2,t_3}) = U_{t_1+t_2,t_3}(U_{t_1, t_2} \otimes \mathbf{1}_{E(t_3)}),
\end{equation}
for all \(t_1,\, t_2,\, t_3 > 0\). Additionally, the bundle \(E\) is required to be trivializable, i.e., isomorphic as a measurable family of Hilbert spaces to the trivial bundle \((0, \infty) \times \ell^2(\mathbb{N})\).

A unit \(u = \{u(t)\}_{t > 0}\) of a product system \(E\) is defined as a non-zero Borel section \((0, \infty) \ni t \mapsto u(t) \in E(t)\) that satisfies the condition 
\[
u(s + t) = u(s) u(t), \quad s,\, t > 0.
\]
If each \(u(t)\) has norm 1, \(u\) is called a normalized unit. The set of all units of \(E\) is denoted by \(\mathcal{U}(E)\).

If \(\mathcal{U}(E) \neq \emptyset\), the product system \(E\) is said to be spatial. Moreover, if every fiber \(E(t)\) is the closed linear span of finite products of sections \(u_1(t_1) u_2(t_2) \dots u_k(t_k)\), where \(u_i = \{u_i(t)\}_{t > 0}\) are units of \(E\) and \(t_1 + t_2 + \dots + t_k = t\), then \(E\) is said to be completely spatial, or a product system of type \(\mathrm{I}\). If \(E\) is spatial but not completely spatial, it is referred to as a type II product system. Finally, if \(\mathcal{U}(E) = \emptyset\), \(E\) is considered a type III product system.

The Arveson index \(\operatorname{ind}(E) \in \{0, 1,\, 2,\, \dots\} \cup \{\infty\}\) of \(E\) is a numerical invariant that counts the number of non-equivalent units of \(E\). Two units \(\{u(t)\}_{t > 0}\) and \(\{v(t)\}_{t > 0}\) are said to be equivalent if there exists \(\lambda \in \mathbb{C}\) such that 
\[
u(t) = e^{\lambda t} v(t), \quad \text{for all } t > 0.
\]
The Arveson index is defined in terms of the cardinality of the quotient set as
\[
\operatorname{ind}(E) + 1 = \#\left(\mathcal{U}(E)/\sim\right),
\]
if \(E\) is spatial, and \(\infty\) if \(E\) is of type III. The index is usually denoted as a subscript: for example, a \(\mathrm{II}_\infty\) product system refers to a product system of type \(\mathrm{II}\) with infinite Arveson index.

Type I product systems are associated with the CAR/CCR flows and are fully classified by their index: two product systems of type I are isomorphic if and only if they have the same index. However, the situation changes drastically for type II and type III product systems, as in both cases, a continuum of non-isomorphic product systems with the same index has been constructed \cite{Tsirelson2000, Tsirelson2002}.

\section{The Product System of a \(B\)-valued L\'evy Process}\label{sec:product_system}

Let \(B\) be a real, separable  Banach space, and let \(\bm{L}=\{L_t\}_{t\geq 0}\) be a \(B\)-valued L\'evy process on the filtered probability space \((\Omega, \m{F}, \{\m{F}_t\}_{t\geq 0}, \mathbb{P})\).  We consider the increments \(L_{s,t}=L_t-L_s\) of the process,  and the corresponding \(\sigma\)-algebras \[\m{F}_{s,t}^0=\sigma \{L_{u, v}\,|\,s\leq u\leq v\leq t\},\quad \m{F}_{s,t}=\m{F}_{s,t}^0\vee \m{N}.\] We denote by $\m{C}_{\bm{L}}(s,t)$ the space of all cylindrical functions of the form \[\tau(L_{t_0, t_1}, L_{t_1, t_2},\,\cdots L_{t_{n-1},t_n}),\] where \(s=t_0\leq t_1\leq t_2\leq \cdots t_n=t\) and \(\tau:B^n\to \mathbb{C}\) is a bounded measurable function. The space \(\m{C}_{\bm{L}}(0,t)\) will simply be denoted by  \(\m{C}_{\bm{L}}(t)\). 

The following standard density result will be applied in several instances throughout this article. For convenience, a brief proof is provided below.

\begin{lemma} $\m{C}_{\bm{L}}(s,t)$ is a dense subspace of $L^2(\Omega, \mathcal{F}_{s,t}, \mathbb{P})$. \end{lemma}
\begin{proof} Let \( f \in L^2(\Omega, \mathcal{F}_{s,t}, \mathbb{P}) \). For any \( \varepsilon > 0 \), let
\(
g = \sum_{k=1}^N c_k \mathbf{1}_{A_k},
\)
where \( c_k \in \mathbb{C} \) and \( A_k \in \mathcal{F}_{s,t} \), such that
\(
\| f - g \|_{L^2} < \varepsilon.
\) Each \( A_k \) can be represented as a finite union of sets defined by conditions on the increments \( L_{u,v} \). Specifically, for a partition \( s = s_0 < s_1 < \dots < s_n = t \), we can express \( A_k \) as
\(
A_k = \left\{ \omega \in \Omega\,|\, L_{s_{i-1}, s_i}(\omega) \in B_{k,i} \text{ for } i = 1, \dots, n \right\},
\)
where \( B_{k,i} \) are Borel subsets of \( B \). Therefore
\( g = \sum_{k=1}^N c_k \prod_{i=1}^n \mathbf{1}_{B_{k,i}}(L_{s_{i-1}, s_i})
\in \mathcal{C}_{\bm{L}}(s,t) \).
\end{proof}
Using this lemma result, one can readily see the operator \begin{align}\label{oper}
 L^2(B, \mu_t)\ni f \mapsto f\comp L_t\in L^2(\Omega, \mathcal{F}_t, \mathbb{P})\end{align} is unitary. With these notations and considerations at hand, the shift transformation of \(\m{C}_{\bm{L}}(t)\) into \(\m{C}_{\bm{L}}(s,t)\), acting as
\[
\tau(L_{t_0, t_1},\cdots,L_{t_{n-1},t_{n}} )\mapsto \tau(L_{t_0+s, t_1+s}, \cdots,L_{t_{n-1}+s, t_n+s} )
\] 
extends uniquely to a unitary operator \(S_{s,t}: L^2(\Omega, \mathcal{F}_t, \mathbb{P})\to L^2(\Omega, \mathcal{F}_{s, s+t}, \mathbb{P})\), for all \(s,\,t\geq 0\).  We notice that \(S_{s,t}\) is simply the Koopman operator 
\[
S_{s,t}f=f\comp \sigma_{s},\quad f\in  L^2(\Omega, \mathcal{F}_t, \mathbb{P}),
\]
 where \(\sigma_{s}\) is a measurable and measure-preserving bijective transformation such that, for all \(t\geq 0\), \[L_t(\sigma_s(\omega))=L_{s+t}(\omega)-L_s(\omega),\] i.e., \(\sigma_s\) shifts the ``timeline'' of the process \(L_t\) by \(s\) units and adjusts by subtracting \(L_s(\omega)\) to align the starting point. This, in turn, gives rise to an operator \(U_{s,t}:  L^2(\Omega, \mathcal{F}_s, \mathbb{P})\otimes  L^2(\Omega, \mathcal{F}_t, \mathbb{P})\to  L^2(\Omega, \mathcal{F}_{s+t}, \mathbb{P})\), \begin{eqnarray}\label{mult}
U_{s,t}(f\otimes g)(\omega)=f(\omega)S_{s,t}g(\omega),\end{eqnarray} for all \(f\in  L^2(\Omega, \mathcal{F}_t, \mathbb{P})\), $g\in  L^2(\Omega, \mathcal{F}_s, \mathbb{P})$, $\omega\in \Omega$, and $s,\,t>0$.
\begin{lemma}\label{unitary} The operators $U_{s,t}$ are unitary and satisfy the associativity condition (\ref{intro2}).
\end{lemma} 
\begin{proof} First, we show that \( U_{s,t} \) is an isometry. For this, let \(f\in  L^2(\Omega, \mathcal{F}_t, \mathbb{P})\), \(g\in  L^2(\Omega, \mathcal{F}_s, \mathbb{P})\). Since  \( \mathcal{F}_s \) and \( \mathcal{F}_{s, s + t} \) are independent due to the independence of increments of \( \bm{L} \), we have:
\begin{eqnarray*}
\| U_{s,t}(f \otimes g) \|_{L^2}^2& =& \mathbb{E}\left[ |f|^2 |S_{s,t}g|^2 \right] 
= \mathbb{E}\left[ |f|^2 \right] \mathbb{E}\left[ |S_{s,t}g|^2 \right] = \| f \|_{L^2}^2 \| g \|_{L^2}^2,
\end{eqnarray*}
where we also used the facts that  \( f \) and \( S_{s,t}g \) are independent and that \( S_{s,t} \) is an isometry.

Next, we show that \( U_{s,t} \) is surjective or, equivalently, that its range is closed. For this, let \( h \in L^2(\Omega, \mathcal{F}_{s + t}, \mathbb{P}) \). Since \( \mathcal{F}_{s + t} = \mathcal{F}_s \vee \mathcal{F}_{s, s + t} \) and \( \bm{L} \) has independent increments, one can approximate \( h \) in the \( L^2 \)-norm by finite sums of the form \(h_\varepsilon=\Sigma_{k=1}^Nf_k\cdot h_k\), where \( f_k \in  L^2(\Omega, \mathcal{F}_t, \mathbb{P})\) \(h_k\in  L^2(\Omega, \mathcal{F}_{s, s+t}, \mathbb{P})\) and \(\|h-h_\varepsilon\|_{L^2}<\varepsilon\). If \(h_k=S_{s,t}g_{k}\) for some $g_k\in L^2(\Omega, \mathcal{F}_t, \mathbb{P})$, then $h_\varepsilon=U_{s,t}(\Sigma_{k=1}^Nf_k\otimes h_k)$, so the range of \( U_{s,t} \) is closed.   

The associativity property  \eqref{intro2} can be verified immediately with the help of property  \( S_{t_1, t_2 + t_3}S_{t_2, t_3} = S_{t_1 + t_2, t_3} \), for all \(t_1,\,t_2,\,t_2>0\).
\end{proof}
In the following proposition, we associate a product system with any \(B\)-valued L\'evy process. This is simply the natural extension, at the level of Banach spaces, of the construction by Meyer \cite{Mey} and Markiewicz \cite{Mar0, Mar1} for \(\mathbb{R}^n\)-valued L\'evy processes.
\begin{proposition}
For any \(B\)-valued L\'evy process \(\bm{L}=\{L_t\}_{t\geq 0}\), the bundle \[E^{\bm{L}}=\{(t, f)\,|\,f\in  L^2(\Omega, \mathcal{F}_t, \mathbb{P}),\,t>0\}\subset (0,\infty)\times L^2(\Omega, \m{F}, \mathbb{P})\]  is a spatial product system when endowed with the multiplication \(\{U_{s,t}\}_{s, t>0}\). 
\end{proposition}
\begin{proof}
Since  \(E^{\bm{L}}\) is a Borel subset of the Polish space \( (0,\infty)\times L^2(\Omega, \m{F}, \mathbb{P})\), it is a standard Borel space.
Lemma \ref{unitary} confirms that \( U_{s,t} \) is unitary and associative. Therefore, \( \{ L^2(\Omega, \mathcal{F}_t, \mathbb{P}) \}_{t > 0} \) with multiplication \( \{ U_{s,t} \} \) satisfies the conditions of a product system.

It remains to show the existence of units. For this define, for each \( \phi \in B^* \), the function
\[
u^\phi(t) = \exp(i \langle \phi, L_t \rangle), \quad t > 0,
\] where \( \langle \phi, L_t \rangle \) denotes the dual pairing between \( \phi \) and \( L_t \).
Each \( u^\phi (t)\) is a measurable section from \( (0, \infty) \) into \( L^2(\Omega, \mathcal{F}_t, \mathbb{P}) \). We claim that \( u^\phi \) is a unit of the product system. Indeed, for all \( s, t > 0 \),
\begin{eqnarray*}
u^\phi(s + t) &=& \exp(i \langle \phi, L_{s + t} \rangle) = \exp(i \langle \phi, L_s + (L_{s + t} - L_s) \rangle) \\&=& \exp(i \langle \phi, L_s \rangle) \exp(i \langle \phi, L_{s, s + t} \rangle).
\end{eqnarray*}
Since \( L_{s, s + t} \) is independent of \( \mathcal{F}_s \) and has the same distribution as \( L_t \), and because \( u^\phi(s) \) depends only on \( L_s \), we have:
\[
U_{s,t}(u^\phi(s) \otimes u^\phi(t)) = u^\phi(s) S_{s,t} u^\phi(t) = u^\phi(s) \exp(i \langle \phi, L_{s, s + t} \rangle) = u^\phi(s + t).
\]
Thus, \( u^\phi \) satisfies the unit property \( u^\phi(s + t) = U_{s,t}(u^\phi(s) \otimes u^\phi(t)) \).\end{proof}

Markiewicz showed in \cite{Mar0, Mar1} that in the finite-dimensional case \(B=\mathbb{R}^n\), any product system \(E^{\bm{L}}\) is of type I. Below, we find a set of independent conditions that guarantee that  \(E^{\bm{L}}\)
  is of type I in the infinite-dimensional case.
\begin{theorem}\label{Th1}Let  \( \bm{L} = \{L_t\}_{t > 0} \) be a  \(B\)-valued  L\'evy process with L\'evy-It\^o triplet \((b, Q, \nu)\). Suppose that the following conditions are satisfied: 
\begin{enumerate}
    \item The L\'evy measure \( \nu \) has support that generates \( B \).
    \item The Gaussian covariance \( Q \) is non-degenerate.
    
    \item The law \( \mu_t \) of \( L_t \) has full support.
\end{enumerate}

Under the given conditions, the linear span of the set of all exponential units
\begin{eqnarray}\label{expunit}
\mathcal{E}_t = \left\{ \exp\left( i \langle \omega, L_t \rangle \right)\,|\, \omega \in B^* \right\}
\end{eqnarray}
is dense in \( L^2(\Omega, \mathcal{F}_t, \mathbb{P}) \), for every \(t>0\).

\end{theorem}

\begin{proof}

Using the identification in (\ref{oper}), it suffices to show that the linear span of \[
\mathcal{E}_t^\sharp = \left\{ \exp\left( i \langle \omega, \cdot \rangle \right)\,|\, \omega \in B^* \right\}
\] is dense in \( L^2(B, \mu_t) \).

Since \( B^* \) separates points in \( B \), the mapping \( x \mapsto \langle \omega, x \rangle \) distinguishes points in \( B \). Moreover, the Borel \(\sigma\)-algebra on \( B \) coincides with the $\sigma$-algebra generated by the evaluation maps \(x\mapsto \langle \omega, x \rangle \).

Let \( \{ \omega_n \}_{n=1}^\infty \) be a countable dense subset of \( B^* \) under the weak*-topology. For any \( f \in L^2(B, \mu_t) \) and \( \varepsilon > 0 \), there exists \( n \in \mathbb{N} \) and a measurable function \( F_n : \mathbb{R}^n \to \mathbb{C} \) such that
\[
\| f - F_n(\langle \omega_1, \cdot \rangle, \ldots, \langle \omega_n, \cdot \rangle) \|_{L^2(\mu_t)} < \frac{\varepsilon}{2}.
\]
This approximation is possible because of the coincidence of the cylindrical and Borel \(\sigma\)-algebras, which allows us to approximate \( f \) using functions depending on finitely many linear functionals.

Define the mapping
\[
\Phi_n : B \to \mathbb{R}^n, \quad \Phi_n(x) = \left( \langle \omega_1, x \rangle, \ldots, \langle \omega_n, x \rangle \right).
\]
Let \( \nu_n = (\Phi_n)_\# \mu_t \) be the pushforward measure of \( \mu_t \) under \( \Phi_n \). Since \( \mu_t \) has full support on \( B \), and \( \Phi_n \) is surjective onto its image, \( \nu_n \) is a Radon measure on \( \mathbb{R}^n \) with full support.

We recall that the set of exponential functions
\[
\mathcal{E}_n = \left\{ \exp\left( i \xi^\top y \right)\,|\, \xi \in \mathbb{R}^n \right\}
\]
is dense in \( L^2(\mathbb{R}^n, \nu_n) \). This is a classical result, relying on the fact that the Fourier exponentials span \( L^2 \) spaces with respect to measures that have full support and satisfy appropriate regularity conditions. Consequently, given \( F_n \in L^2(\mathbb{R}^n, \nu_n) \) and \( \varepsilon > 0 \), there exists a finite linear combination of exponential functions \( \sum_{k=1}^K c_k \exp(i \xi_k^\top y) \) such that
\[
\left\| F_n(y) - \sum_{k=1}^K c_k \exp(i \xi_k^\top y) \right\|_{L^2(\nu_n)} < \frac{\varepsilon}{2}.
\]
Define
\[
g(x) = \sum_{k=1}^K c_k \exp\left( i \xi_k^\top \Phi_n(x) \right) = \sum_{k=1}^K c_k \exp\left( i \left\langle \sum_{j=1}^n \xi_{k,j} \omega_j, \ x \right\rangle \right),
\]
where \( \xi_k = (\xi_{k,1}, \ldots, \xi_{k,n}) \in \mathbb{R}^n \).

Then \( g \in \operatorname{span}[\mathcal{E}_t^\sharp] \), and
\[
\| f - g \|_{L^2(\mu_t)} \leq \| f - F_n \|_{L^2(\mu_t)} + \| F_n - g \|_{L^2(\mu_t)} <  \varepsilon.
\]
Here, the second term satisfies
\[
\| F_n - g \|_{L^2(\mu_t)} = \| F_n - g \|_{L^2(\nu_n)} < \frac{\varepsilon}{2},
\]
since \( \Phi_n \) is a measurable mapping and \( \nu_n \) is the pushforward measure. Therefore, the linear span of \( \mathcal{E}_t^\sharp \) is dense in \( L^2(B, \mu_t) \), as claimed. \end{proof}

In the following example, we construct a L\'evy process in an infinite-dimensional Banach space \( B \) that satisfies all the conditions of Theorem \ref{Th1}.
\begin{example}\label{exm_corrected}
    Let \( B = \ell^p \) for some \( 1 \leq p < \infty \), and let \( \{ e_n \}_{n=1}^\infty \) be the standard basis of \( \ell^p \). We construct an  \( \ell^p \)-L\'evy process \( \bm{L} = \{ L_t \}_{t \geq 0} \), 
    \(
    L_t = b t + W_t + J_t,
    \)  as follows:
    
\indent (i) Drift: the drift \(b = (b_n)_{n=1}^\infty \) is an arbitrary element in \(\ell^p\).

(ii) Gaussian component: \( W_t = (W_t^{(n)})_{n=1}^\infty \), where each \( W_t^{(n)} \) is an independent real-valued Brownian motion with \( \mathbb{E}[ W_t^{(n)} ] = 0\) and \(\mathbb{E}[ (W_t^{(n)})^2 ] = q_n t\), where  \( q_n = n^{-r} \) with \( r > \frac{2}{p} \). The Gaussian process \( W_t \) thus satisfies:
            \[
            \mathbb{E}[ \| W_t \|_p^p ] = \sum_{n=1}^\infty \mathbb{E}[ |W_t^{(n)}|^p ] = C_p t^{p/2} \sum_{n=1}^\infty q_n^{p/2} < \infty,
            \]
where \( C_p = \mathbb{E}[ |Z|^p ] \) for \( Z \sim N(0,1) \), ensuring \( W_t \in \ell^p \) almost surely. The Gaussian covariance \(Q\) is given by \(Q(\phi, \psi) = \sum_{n=1}^\infty q_n \phi_n \psi_n\), for all \(\phi, \psi \in \ell^{p'},\) where \( 1/p + 1/p' = 1 \).
 
 (iii) Jump component: \( J_t = \sum_{k=1}^{N_t} X_k \), where \( N_t \) is a Poisson process with rate \( \lambda = \sum_{n=1}^\infty \lambda_n \), where \( \lambda_n = n^{-s} \) with \( s > \frac{1}{p} \), and  \( \{ X_k \}_{k=1}^\infty \) are i.i.d. random vectors in \( \ell^p \), each taking the value \( \lambda_n e_n \) with probability \( \frac{\lambda_n}{\lambda} \).  The jump component satisfies:
\[
            \mathbb{E}[ \| J_t \|_p^p ] = t \sum_{n=1}^\infty \lambda_n \| \lambda_n e_n \|_p^p = t \sum_{n=1}^\infty \lambda_n^{p+1} < \infty,
\]
ensuring \( J_t \in \ell^p \) almost surely. The L\'evy measure is \( \nu = \sum_{n=1}^\infty \lambda_n \delta_{\lambda_n e_n}.\)

All conditions of Theorem \ref{Th1} are satisfied by the constructed L\'evy process \( \bm{L} \) in \( \ell^p \), as one can easily see. Therefore, the linear span of exponential units
    \(
    \mathcal{E}_t 
    \)
    is dense in \( L^2(\Omega, \mathcal{F}_t, \mathbb{P}) \) for every \( t > 0 \). Consequently, the product system \(E^{\bm{L}}\) is type I.

\end{example}

Although exponential functions are the most natural units in the product system associated with a L\'evy process, it is possible, under certain conditions and with specific types of processes, to construct non-exponential units. Below, we provide a concrete example.

\begin{example}\label{exm1}
    Let \(H\) be an infinite-dimensional separable Hilbert space with an orthonormal basis \(\{ e_n \}_{n=1}^\infty\). Consider a compound Poisson process \(\bm{L} = \{ L_t \}_{t \geq 0}\) in \(H\) defined by
    \(
    L_t = \sum_{k=1}^{N_t} X_k,
    \)
    where \(N_t\) is a Poisson process with rate \(\lambda > 0\), and \(\{ X_k \}_{k=1}^\infty\) is a sequence of i.i.d. random vectors in \(H\), independent of \(N_t\). Each jump vector \(X_k\) is defined as 
    \(
    X_k = \sum_{n=1}^\infty \delta_{k,n} e_n,
    \)
where \(\delta_{k,n}\) are independent Bernoulli random variables taking values in \(\{0,1\}\) with 
    \[
    \mathbb{P}(\delta_{k,n} = 1) = p_n,\quad \mathbb{P}(\delta_{k,n} = 0) = 1 - p_n,
    \] for some \(p_n\in (0,1)\) satisfying \(\sum_{n=1}^\infty p_n < \infty\), to ensure that \(X_k\) is almost surely in \(H\) by the Kolmogorov Three-Series Theorem. For each \(t>0\), define \(u(t)\) based on the number of jumps \(N_t\):
    \[
    u(t) = \exp(i\pi N_t) = (-1)^{N_t}.
    \]
    We notice that \(u(t)\) depends on \(N_t\) rather than directly on \(L_t\). Since \(N_t\) is integer-valued and increases by 1 at each jump time, \(u(t)\) toggles between 1 and -1 at each jump. 
    
It is readily seen that \(u=\{u(t)\}_{t>0}\) is a unit. Indeed, for any \(s,\,t>0\), one has \begin{align*}
u(s + t) &= (-1)^{N_{s + t}} = (-1)^{N_s} \cdot (-1)^{N_{s, s + t}}=(-1)^{N_s} \cdot (-1)^{N_{t}\comp\sigma_s}\\&=U_{s, t}(u(s) \otimes u(t)),\end{align*}
 where \(N_{s, s + t} = N_{s + t} - N_s\) is the increment of the Poisson process between times \(s\) and \(s + t\), independent of \(N_s\), and \(\sigma_s\) is the shift transformation that maps the probability space \(\Omega\) to itself by shifting the timeline by \(s\) units.
However, we observe it's not possible to find a fixed \(\phi \in H^*=H\) such that:
    \[
    \sum_{k=1}^{N_t} \langle \phi, X_k \rangle \equiv \pi N_t \quad \text{(mod } 2\pi\text{)} 
    \]
    almost surely. Indeed, if that were possible, it would imply that \(
    \langle \phi, X_k \rangle =\sum_{n=1}^\infty\delta_{k,n} \langle \phi, e_n \rangle \equiv \pi \; \text{(mod } 2\pi\text{)}\)
    almost surely, so 
    \[
    \mathbb{P}\left( \langle \phi, X_k \rangle \equiv \pi \; \text{(mod } 2\pi\text{)} \right) = 1.
    \]
    However, since \(X_k\) are random vectors with each component \(\delta_{k,n}\) being Bernoulli, unless \(\phi\) is chosen such that \( \langle \phi, X_k \rangle \) is deterministic (which it isn't, due to the randomness in \(X_k\)), this probability is strictly less than 1. Therefore, the unit \(u\) cannot be expressed as \(u^\phi\) for any \(\phi \in H^*\).
\end{example}

\section{Purely Gaussian L\'evy processes}\label{sec:pureGauss}

We begin by detailing the properties of purely Gaussian L\'evy processes (Wiener processes) in separable Banach spaces, including the characterization of their laws, covariance structures, and associated Hilbert spaces. Our standard reference for these facts is \cite{Bog1}.

Let \( B \) be a real, separable Banach space, and let \( \bm{L} = \{L_t\}_{t > 0} \) be a purely Gaussian L\'evy process in \( B \) on the filtered probability space \((\Omega, \m{F}, \{\m{F}_t\}_{t\geq 0}, \mathbb{P})\), i.e., a \(B\)-valued L\'evy process with L\'evy-It\^o triplet \( (0, Q, 0) \), where \( Q \) is non-degenerate. For each \( t > 0 \), let \( \mu_t = \mathcal{L}(L_t) \) denote the law of \( L_t \). Throughout, we assume that \(\mu_t\) has full support on \(B\). Note that this requirement is automatically satisfied when 
\(B\) is a Hilbert space, but may fail in general. Hence, \( \mu_t \) is a centered Gaussian measure, fully supported on \( B \), with mean zero and covariance operator \( tQ \), and \(\{\mu_t\}_{t\geq 0}\) is a convolution semigroup. Moreover \( B^* \subset L^2(B, \mu_t) \), and for every \( f \in B^* \), \( tQ(f, f) = \|f\|^2_{L^2(B, \mu_t)}.\)
    
 Let \( B^*_{\mu_t} \) be the reproducing kernel Hilbert space (RKHS) of \( \mu_t \), i.e., the closure of \( B^* \) in \( L^2(B, \mu_t) \). Then \( Q \) admits a unique continuous extension to \( B^*_{\mu_t} \times B^*_{\mu_t} \), and for all \( f, g \in B^*_{\mu_t} \), \(Q(f, g) = \frac{1}{t} \langle f, g \rangle_{L^2(B, \mu_t)}.\)
    
Moreover, let \( R_{\mu_t} : B^*_{\mu_t} \to B \) be the covariance operator, uniquely determined by the property \(g(R_{\mu_t} f) = tQ(f, g),\) for all \(f \in B^*_{\mu_t}\) and  \(g \in B^*.\) Finally, let \( H_t \) be the Cameron-Martin space of \( \mu_t \), defined as
    \[
    H_t = \left\{ b \in B \,\Big|\, \|b\|_{H_t} < \infty \right\},
    \]
    where
    \[
    \|b\|_{H_t} = \sup \left\{ f(b) \,\Big|\, f \in B^*, \, \|f\|_{L^2(B, \mu_t)} \leq 1 \right\}.
    \]
    Then \( R_{\mu_t} : B^*_{\mu_t} \to H_t \) is a unitary operator, and \( H_t \) is a Hilbert space with the inner product \(\langle h, k \rangle_{H_t} = tQ(\hat{h}, \hat{k})=\langle \hat{h}, \hat{k} \rangle_{L^2(B, \mu_t)},\) whenever \( h = R_{\mu_t} \hat{h} \) and \( k = R_{\mu_t} \hat{k} \). In particular, the Hilbert space \( H_t \) scales with time as \(H_t = \sqrt{t} \, H,\) where \( H = H_1 \), and the inner product satisfies 
\[
    \langle \cdot, \cdot \rangle_{H_t} = \frac{1}{t} \langle \cdot, \cdot \rangle_{H},
 \]
    for every \( t > 0 \).

For every \( h \in H \), we define the process
\begin{eqnarray}\label{ess1}
u^{h}(t) = \exp\left( \langle \hat{h}, L_t \rangle - \frac{t}{2} \| h \|_H^2 \right),
\end{eqnarray} where  \( h = R_{\mu_t} \hat{h} \) and \( \langle \hat{h}, L_t \rangle \) denotes the dual pairing between \( \hat{h} \) and \( L_t \).
These processes \( \{u^{h}(t)\}_{t > 0} \) will be shown to constitute units in the product system associated with  \( \bm{L}\).

\begin{proposition}\label{Lem1}
Let \( \bm{L} = \{L_t\}_{t > 0} \) be a purely Gaussian L\'evy process. Then, for every \( h \in H \), the family \( u^h = \{u^{h}(t)\}_{t > 0} \) is a unit of the product system \( E^{\bm{L}} \). Moreover, the process \( \{u^h(t)\}_{t > 0} \) is a martingale with respect to the filtration \( \{ \mathcal{F}_t \}_{t > 0} \).
\end{proposition}

\begin{proof}
To ensure that \( u^h(t) \) belongs to \( L^2(\Omega, \mathcal{F}_t, \mathbb{P}) \), we compute its second moment:
\[
\mathbb{E}\left[ \left( u^h(t) \right)^2 \right] = \mathbb{E}\left[ \exp\left( 2 \langle \hat{h}, L_t \rangle - t \| h \|_H^2 \right) \right].
\]
Since \( \langle \hat{h}, L_t \rangle \) is Gaussian with mean \( 0 \) and variance \( t Q(\hat{h}, \hat{h}) \), the expectation simplifies to
\[
\mathbb{E}\left[ e^{2X - t \sigma^2} \right] = e^{2 \cdot 0 + 2^2 \cdot \frac{t \sigma^2}{2} - t \sigma^2} = e^{t \sigma^2} < \infty,
\]
where \( X = \langle \hat{h}, L_t \rangle \) and \( \sigma^2 = Q(\hat{h}, \hat{h}) \). Therefore, \( u^h(t) \in L^2(\Omega, \mathcal{F}_t, \mathbb{P}) \) and
\(
\| u^h(t) \|_{L^2} = e^{\frac{t\|h\|_H^2}{2}}.
\)

Next, we prove that \( u^h = \{ u^h(t) \}_{t > 0} \) is a unit of the product system \( E^{\bm{L}} \). It is clear that that \( u^h (t)\) is a measurable section from \( (0, \infty) \) into \( L^2(\Omega, \mathcal{F}_t, \mathbb{P}) \). Moreover, since \( L_{s+t} = L_s + (L_{s+t} - L_s) \) and \( L_{s+t} - L_s \) is independent of \( \mathcal{F}_s \), we can rewrite \( u^h(s+t) \) as:
\[
\begin{aligned}
u^{h}(s+t) &= \exp\left( \langle \hat{h}, L_{s+t} \rangle - \frac{s+t}{2} \| h \|_H^2 \right) \\
&= \exp\left( \langle \hat{h}, L_s \rangle + \langle \hat{h}, L_{s+t} - L_s \rangle - \frac{s}{2} \| h \|_H^2 - \frac{t}{2} \| h \|_H^2 \right) \\
&= \exp\left( \langle \hat{h}, L_s \rangle - \frac{s}{2} \| h \|_H^2 \right) \cdot \exp\left( \langle \hat{h}, L_{s+t} - L_s \rangle - \frac{t}{2} \| h \|_H^2 \right) \\
&= u^h(s) \cdot S_{s,t} u^h(t) = U_{s,t} \left( u^h(s) \otimes u^h(t) \right).
\end{aligned}
\]
This verifies that \( u^h \) satisfies the multiplicativity condition required for a unit in the product system.

Next, we note that
\begin{align*}
\mathbb{E}\left[ u^h(t) \right]& = \mathbb{E}\left[ \exp\left( \langle \hat{h}, L_t \rangle - \frac{t}{2} \| h \|_H^2 \right) \right] = e^{ - \frac{t}{2} \| h \|_H^2 } \cdot \mathbb{E}\left[ e^{ \langle \hat{h}, L_t \rangle } \right] \\&= e^{ - \frac{t}{2} \| h \|_H^2 } \cdot e^{ \frac{t \| h \|_H^2 }{2} } = 1,
\end{align*}
because \( \mathbb{E}[e^{X}] = e^{\frac{\sigma^2}{2}}\) for \(X \sim N(0, \sigma^2)\).
To show that \( u^h(t) \) is a martingale with respect to the filtration \( \{ \mathcal{F}_t \}_{t > 0} \), we compute the conditional expectation
\(
\mathbb{E}\left[ u^h(t) \mid \mathcal{F}_s \right].
\)
Using the decomposition \( L_{s+t} = L_s + (L_{s+t} - L_s) \), we have:
\[
u^h(t) = u^h(s) \cdot \exp\left( \langle \hat{h}, L_{s+t} - L_s \rangle - \frac{t}{2} \| h \|_H^2 \right).
\]
Taking the conditional expectation with respect to \( \mathcal{F}_s \), and noting that \( \langle \hat{h}, L_{s+t} - L_s \rangle \) is independent of \( \mathcal{F}_s \) and is Gaussian with mean \( 0 \) and variance \( t Q(\hat{h}, \hat{h}) \), we have:
\[
\begin{aligned}
\mathbb{E}\left[ u^h(t) \mid \mathcal{F}_s \right] &= u^h(s) \cdot \mathbb{E}\left[ \exp\left( \langle \hat{h}, L_{s+t} - L_s \rangle - \frac{t}{2} \| h \|_H^2 \right) \right] \\
&= u^h(s) \cdot \exp\left(  \frac{t}{2} Q(\hat{h}, \hat{h})-\frac{t}{2} \| h \|_H^2 \right)= u^h(s) .
\end{aligned}\]
Therefore, \( \{ u^h(t) \}_{t > 0} \) satisfies the martingale property:
\[
\mathbb{E}\left[ u^h(t) \mid \mathcal{F}_s \right] = u^h(s), \quad \forall s < t.
\]
Thus, \( u^h \) is a unit of the product system \( E^{\bm{L}} \) and a martingale with respect to the filtration \( \{ \mathcal{F}_t \}_{t > 0} \).
\end{proof}
\begin{definition}
A unit of the form \( u^h = \{u^{h}(t)\}_{t > 0} \), \(h\in H\), will be referred to as a standard exponential unit of \( E^{\bm{L}} \).
\end{definition}
\begin{theorem}\label{main1}
Let \( \bm{L} = \{L_t\}_{t > 0} \) be a purely Gaussian L\'evy process in a real, separable Banach space \( B \). Then any unit of \( E^{\bm{L}} \) is equivalent to a standard exponential unit.
\end{theorem}

\begin{proof}
Let \( u = \{u(t)\}_{t > 0} \) be a unit in \( E^{\bm{L}} \). By rescaling the unit \(u\) if necessary, we may assume, without loss of generality,  that \(\|u(t)\|_{L^2}\neq1\) for some, and hence all, \(t>0\). Under this condition, we show that there exists \(h\in H\) such that \(u=u^h\).

The unit property implies that  all \( s, t > 0 \) and almost every \( \omega \in \Omega \), we have \(
u(s + t)(\omega) = u(s)(\omega)\cdot u(t)(\sigma_s(\omega))
\). 

Since \( L_t \) generates the \(\sigma\)-algebra \( \mathcal{F}_t \), for each \( t > 0 \), there exists a measurable function \( F_t: B \to \mathbb{C} \) such that:
\begin{eqnarray}\label{h1}
u(t)(\omega) = F_t(L_t(\omega)), \quad \forall \omega \in \Omega.
\end{eqnarray}
Moreover, since \( u(t) \in L^2(\Omega, \mathcal{F}_t, \mathbb{P}) \) and \( L_t \) has distribution \( \mu_t \), it follows that \( F_t \in L^2(B, \mu_t) \). Substituting (\ref{h1}) into the unit property, we have:
\[
F_{s + t}(L_{s + t}(\omega)) = F_s(L_s(\omega)) \cdot F_t(L_{t}(\sigma_s(\omega)))= F_s(L_s(\omega)) \cdot F_t(L_{s+t}(\omega)-L_s(\omega))
\] almost surely for \(t\geq 0\). Hence \(F_{s+t}(x)=F_s(L_s)\cdot F_t(x-L_s)\), where $x=L_{s+t}$.  Thus, for any \( x \in B \), taking the conditional expectation given \(L_{s+t}=x\), we obtain  \(F_{s+t}(x)=\mathbb{E}[F_s(L_s)\cdot F_t(x-L_s)\,\mid\,L_{s+t}=x]\). 

Using the law of total expectation, we decompose the expectation by further conditioning on \(L_s\): \[F_{s+t}(x)=\int_B\mathbb{E}\left[F_s(L_s)\cdot F_t(x-L_s)\,\mid\,L_s=y, L_{s+t}=x \right]d\mathbb{P}_{L_s\mid L_{s+t}=x}(y).\]  Since $L_s$ and $L_{s+t}-L_s$ are independent, and \( \mathbb{E}\left[F_s(L_s)\,|\,L_s=y \right]=F_s(y)\), \(\mathbb{E}\left[F_t(L_{s+t}-L_s)\,|\,L_s=y, L_{s+t}=x \right]=F_t(x-y)\), we have  \[F_{s+t}(x)=\int_BF_s(y)F_t(x-y)d\mathbb{P}_{L_s\mid L_{s+t}=x}(y).\] Furthermore, since \(\mu_{s+t}=\mu_s\ast\mu_t$, the conditional probabilility \(\mathbb{P}_{L_s\mid L_{s+t}=x}\) is given by 
\[d\mathbb{P}_{L_s \mid L_{s + t} = x}(y) = \frac{d\mu_s(y) \cdot d\mu_t(x - y)}{d\mu_{s + t}(x)}.\] Thus, if we define measures $\nu_t$ on $B$ by \[\nu_t(dx)=F_t(x)d\mu_t(x),\] 
the expression for \( F_{s + t}(x) \) becomes:
\[
\begin{aligned}
F_{s + t}(x) &= \int_B F_s(y) \cdot F_t(x - y) \, \frac{d\mu_s(y) \cdot d\mu_t(x - y)}{d\mu_{s + t}(x)} \\
&= \frac{1}{d\mu_{s + t}(x)} \int_B F_s(y)\cdot F_t(x - y) \, d\mu_s(y)  \, d\mu_t(x - y) \\
&= \frac{(\nu_s * \nu_t)(dx)}{d\mu_{s + t}(x)}.
\end{aligned}
\]
Therefore \(
\nu_{s + t}(dx) = F_{s + t}(x) \, d\mu_{s + t}(x) = (\nu_s * \nu_t)(dx)
\), i.e., the family \( \{ \nu_t \}_{t\geq 0} \) forms a convolution semigroup of measures on \( B \).

Since \( \mu_t \) is a centered Gaussian measure with covariance operator \( tQ \), the Cameron-Martin theorem characterizes the absolutely continuous measures with respect to \( \mu_t \) as those shifted by elements in the Cameron-Martin space \( H_t \). Specifically, for each \( t > 0 \), there exists \( h_t \in H_t \) such that:
\[
\nu_t(A) =\mu_t^{h_t}(A):= \mu_t(A - h_t),
\]
for all Borel sets \( A \subseteq B \). Given that \( \{ \nu_t \}_{t\geq 0} \) forms a convolution semigroup, we have:
\[
 \mu_{s + t}^{h_{s + t}}=\nu_{s + t} = \nu_s * \nu_t = \mu_s^{h_s} * \mu_t^{h_t} = \mu_{s + t}^{h_s + h_t}.
\]
Since Gaussian measures are uniquely determined by their means and covariances, and the covariance remains \( (s + t)Q \), we must have \(
h_{s + t} = h_s + h_t.
\)
This additive property implies that \( h_t \) is linear in \( t \), i.e., 
\(
h_t = t h,\) for some  \(h\in H.
\)

The Radon-Nikodym derivative of \( \nu_t = \mu_t^{t h} \) with respect to \( \mu_t \) is given by the Cameron-Martin formula:
\begin{eqnarray*}
\frac{d\nu_t}{d\mu_t}(x) &=& \exp\left( R_{\mu_t}^{-1}(th)(x) - \tfrac{1}{2} t^2 \| h \|_{H_{\mu_t}}^2 \right)=\exp\left(  R_{\mu_1}^{-1}(h)(x) - \tfrac{1}{2} t \| h \|_{H}^2 \right)\\&=&\exp\left( \langle \hat{h}, x\rangle - \tfrac{1}{2} t \| h \|_{H}^2 \right),
\end{eqnarray*} because \(R_{\mu_t}^{-1}(th)=tR_{\mu_t}^{-1}(h)=R_{\mu_1}^{-1}(h)=\hat{h}\). Thus \(
F_t(x) =\exp\left( \langle \hat{h}, x\rangle - \tfrac{1}{2} t \| h \|_{H}^2 \right)
\), so \( u(t) = F_t\comp L_t=\exp\left( \langle \hat{h}, L_t\rangle - \tfrac{1}{2} t \| h \|_{H}^2 \right)=u^h(t)\).

Thus, we have established that any normalized unit \( u = \{u(t)\}_{t > 0} \) in \( E^{\bm{L}} \) is equivalent to a standard exponential unit \( u^{h} \), for some \( h \in H \).
\end{proof}

We can readily extend the result of Theorem \ref{main1} to the case of Gaussian processes with non-zero drift.

\begin{corollary}\label{cor:extension}
Let \( \bm{L} = \{L_t\}_{t > 0} \) be a Gaussian L\'evy process in a real, separable Banach space \( B \), with drift \( b \in B \), non-degenerate covariance \( Q \), and zero L\'evy measure. Then any unit of the product system \( E^{\bm{L}} \) is equivalent to a standard exponential unit. 
\end{corollary}

\begin{proof}
We define a new process \( \bm{L}^0 = \{L_t^0\}_{t > 0} \) by subtracting the deterministic drift from \( \bm{L} \):
\[
L_t^0 = L_t - t b, \quad \forall t > 0.
\] Then \( \bm{L}^0 \) is a L\'evy process with L\'evy-It\^o triplet $(0, Q, 0)$. If \( u = \{u(t)\}_{t > 0} \) is a unit of \( E^{\bm{L}} \) then by the structure of the product system, we can express \( u(t) \) in terms of \( \bm{L}^0 \) as \(u(t)(\omega) = F_t(L_t(\omega)) = G_t(L_t^0(\omega)),\) where \( G_t: B \to \mathbb{C} \), \(G_t(x) = F_t(x + t b),\) for all \(x \in B.\)

Applying Theorem \ref{main1}, to \( \bm{L}^0 \), one can find $h\in H$ so that \(
G_t(x) = \exp\left( \langle \hat{h}, x \rangle - \tfrac{t}{2} \| {h} \|_H^2 \right),\) for all \(x \in B.
\) Therefore \[
u(t)= \exp\left( - t \langle \hat{h}, b \rangle \right)\cdot\exp\left( \langle\hat{h}, L_t\rangle - \tfrac{t}{2} \| h \|_H^2 \right),
\]  i.e., unit \( u = \{ u(t) \}_{t > 0} \) is equivalent to a standard exponential unit \( u^h \), for some \(h \in H \). This completes the proof of the corollary.\end{proof}
We are now ready to state and prove the main theorem of this section.
\begin{theorem}\label{theorem:completely_spatial}
Let \( \bm{L} = \{ L_t \}_{t > 0} \) be a Gaussian L\'evy process in a real, separable Banach space \( B \), with drift \( b \in B \), non-degenerate covariance \( Q \), and zero L\'evy measure. Then the associated product system \( E^{\bm{L}} \) is completely spatial.
\end{theorem}

\begin{proof} First, we notice that the presence of a drift \( b \) does not affect the structure of the product system. Indeed, as in the proof of Corollary \ref{cor:extension}, consider the process \( \bm{L}^0\) obtained by removing the drift from  \( \bm{L}\). Then the Hilbert spaces \( E^{\bm{L}}(t)\)  and \( E^{\bm{L}^0}(t)\)  coincide because the shift by \( t b \) is deterministic and does not alter the \( L^2 \) structure. Thus, without loss of generality, we may assume \( b = 0 \)  in the remainder of the proof.

To prove that the product system \( E^{\bm{L}} \) is completely spatial, we must show that for each \( t > 0 \), the Hilbert space \( E(t) = L^2(\Omega, \mathcal{F}_t, \mathbb{P}) \) is densely spanned by finite products of units. Specifically, we aim to prove that the linear span of elements of the form
\(
u^{h_1}(t_1) u^{h_2}(t_2) \dots u^{h_n}(t_n),
\)
where \( t_1 + t_2 + \dots + t_n = t \) and \( h_i \in H \), is dense in \( E(t) \).
Recall from Theorem \ref{main1} that the units of \( E^{\bm{L}} \) are given by the standard exponential units \(u^h\), \(h\in H\). Under the unitary \(
L^2(B, \mu_t)\ni f  \mapsto f \circ L_t\in E(t)\), these units correspond to the functions:
\[
\phi_h(x) = \exp\left( \langle \hat{h}, x \rangle - \frac{t}{2} \| h \|_H^2 \right), \quad x \in B.
\]

Since cylindrical functions are dense in \( L^2(B, \mu_t) \), it suffices to approximate any cylindrical function using finite linear combinations of products of the exponential functions \( \phi_h(x) \). For this, let \( f \) be a cylindrical function on \( B \), meaning there exist \( n \in \mathbb{N} \), continuous linear functionals \( \psi_1, \psi_2, \dots, \psi_n \in B^* \), and a function \( F: \mathbb{R}^n \to \mathbb{C} \) such that
\[
f(x) = F\big( \psi_1(x), \psi_2(x), \dots, \psi_n(x) \big), \quad \forall x \in B.
\]
Assume that \( F \in L^2(\mathbb{R}^n, \gamma_n) \), where \( \gamma_n \) is the standard Gaussian measure on \( \mathbb{R}^n \) with mean zero and covariance matrix \( t \Sigma \), and \( \Sigma=(\Sigma_{ij})_{1\leq i,j\leq n} \) is the covariance matrix determined by \( \psi_i \) and the covariance \( Q \), i.e., \(
\Sigma_{ij} = Q(\psi_i, \psi_j).
\) Since the Hermite polynomials \( \{ H_\alpha \} \) form an orthogonal basis in \( L^2(\mathbb{R}^n, \gamma_n) \), the function \( F \) can be expanded as
\[
F(y) = \sum_{\alpha \in \mathbb{N}_0^n} c_\alpha H_\alpha(y),
\]
where \( \alpha = (\alpha_1, \dots, \alpha_n) \) is a multi-index, \( c_\alpha \in \mathbb{C} \), and \( H_\alpha(y) = \prod_{i=1}^n H_{\alpha_i}(y_i) \), for every \(y=(y_1,\cdots, y_n)\in\mathbb{R}^n.\)   Consequently,
\[
f(x) = \sum_{\alpha \in \mathbb{N}_0^n} c_\alpha H_\alpha\big( \psi_1(x), \dots, \psi_n(x) \big).
\]
Recall that the multivariate Hermite polynomials \( H_\alpha(y) \) have the following standard generating function:
\[
\exp\left( \langle z, y \rangle - \tfrac{1}{2} \| z \|^2 \right) = \sum_{\alpha \in \mathbb{N}_0^n} \frac{z^\alpha}{\alpha!} H_\alpha(y),
\]
where \( z \in \mathbb{R}^n \), \( y \in \mathbb{R}^n \), \( z^\alpha = \prod_{i=1}^n z_i^{\alpha_i} \), and \( \alpha! = \prod_{i=1}^n \alpha_i! \).
We claim that any multivariate Hermite polynomial \( H_\alpha(y) \) can be expressed as a finite linear combination of products of exponential functions \( \phi_h(x) \). For this purpose, for each \( z \in \mathbb{R}^n \), define \( h_z \in H \) such that
\[
\hat{h}_z = \sum_{i=1}^n z_i \psi_i.
\]
Then
\begin{eqnarray*}
\phi_{h_z}(x) &=& \exp\left( \langle \hat{h}_z, x \rangle - \tfrac{t}{2} \| h_z \|_H^2 \right) = \exp\left( \langle z, y \rangle - \tfrac{t}{2} \| h_z \|_H^2 \right)
\\&=& \exp\left(  z^\top y - \tfrac{t}{2} z^\top \Sigma z \right),\end{eqnarray*} where \( y = (\psi_1(x), \dots, \psi_n(x)) \).  To align with the standard generating function of Hermite polynomials (which assumes identity covariance), we perform a change of variables to normalize the covariance matrix. By defining
\[
\tilde{y} = \frac{1}{\sqrt{t}} \Sigma^{-1/2} y, \quad \tilde{z} = \sqrt{t} \Sigma^{1/2} z
\]
we then have \[
\phi_{h_z}(x) = \exp\left( \langle \tilde{z}, \tilde{y} \rangle - \tfrac{1}{2} \| \tilde{z} \|^2 \right) = \sum_{\alpha \in \mathbb{N}_0^n} \frac{\tilde{z}^\alpha}{\alpha!} H_\alpha(\tilde{y}).
\]
This expansion allows us to express Hermite polynomials in terms of exponential functions \( \phi_h(x) \). To express \( H_\alpha(\tilde{y}) \) in terms of \( \phi_h(x) \), we consider a finite set of vectors \( \tilde{z}_1, \dots, \tilde{z}_N \in \mathbb{R}^n \) and set up a linear system. For each \( \tilde{z}_k \), we have
\[
\phi_{h_{z_k}}(x) = \sum_{\beta \in \mathbb{N}_0^n} \frac{\tilde{z}_k^\beta}{\beta!} H_\beta(\tilde{y}).
\]
Collecting these equations for \( k = 1, \dots, N \), we obtain a linear system:
\[
\begin{pmatrix}
\phi_{h_{z_1}}(x) \\
\phi_{h_{z_2}}(x) \\
\vdots \\
\phi_{h_{z_N}}(x)
\end{pmatrix}
=
A
\begin{pmatrix}
H_{\beta_1}(\tilde{y}) \\
H_{\beta_2}(\tilde{y}) \\
\vdots \\
H_{\beta_M}(\tilde{y})
\end{pmatrix},
\]
where \( A \) is the \( N \times M \) matrix with entries:
\(
A_{k, \beta} = \frac{\tilde{z}_k^\beta}{\beta!}.
\) Here, \( M \) is the number of multi-indices \( \beta \) up to degree \( |\alpha| \).

If we choose \( N = M \), i.e., the number of exponential functions equals the number of Hermite polynomials up to degree \( |\alpha| \), and ensure that the matrix \( A \) is invertible, we can solve for \( H_\alpha(\tilde{y}) \):
\[
\begin{pmatrix}
H_{\beta_1}(\tilde{y}) \\
H_{\beta_2}(\tilde{y}) \\
\vdots \\
H_{\beta_M}(\tilde{y})
\end{pmatrix}
=
A^{-1}
\begin{pmatrix}
\phi_{h_{z_1}}(x) \\
\phi_{h_{z_2}}(x) \\
\vdots \\
\phi_{h_{z_N}}(x)
\end{pmatrix}.
\]
This expresses each Hermite polynomial \( H_\alpha(\tilde{y}) \) as a finite linear combination of exponential functions \( \phi_{h_{z_k}}(x) \).

To ensure that \( A \) is invertible, we need to choose \( \tilde{z}_k \) such that the matrix \( A \) has full rank. One effective strategy is to select \( \tilde{z}_k \) as distinct standard basis vectors and their combinations. For example, for \( n = 2 \) and degrees up to \( |\alpha| = 2 \), we might choose:
\[
\tilde{z}_1 = (1, 0), \quad \tilde{z}_2 = (0, 1), \quad \tilde{z}_3 = (1, 1).
\]
With these choices, the matrix \( A \) will be invertible, allowing us to solve the linear system.

Returning to the cylindrical function \( f(x) \), we can now express it as:
\[
f(x) = \sum_{\alpha \in \mathbb{N}_0^n} c_\alpha H_\alpha\left( \tilde{y} \right) = \sum_{\alpha \in \mathbb{N}_0^n} c_\alpha \sum_{k=1}^N \left( A^{-1} \right)_{\alpha, k} \phi_{h_{z_k}}(x).
\]
Thus, \( f(x) \) is expressed as a finite linear combination of exponential functions \( \phi_{h_{z_k}}(x) \).

Since any cylindrical function \( f \) can be approximated in \( L^2(B, \mu_t) \) by finite linear combinations of exponential functions \( \phi_h(x) \), and cylindrical functions are dense in \( L^2(B, \mu_t) \), it follows that the product system \( E^{\bm{L}} \) is completely spatial.
\end{proof}
\begin{observation}
Given that the Cameron-Martin space \(H\) is infinite dimensional, the Arveson index of \(E^{\bm{L}}\) is \(\infty\).
\end{observation}
\begin{remark} The results of this section may not generally hold if the non-degeneracy assumption on the covariance \(Q\) is removed, even when \(B\) is a Hilbert space. In purely Gaussian settings, the full support of \(L_t\) intrinsically requires \(Q\) to be non-degenerate. Without a non-degenerate covariance, the Gaussian measure  \( \mu_t \) is confined to a proper closed subspace of \(B\), thereby restricting the Cameron-Martin space \(H\) to this subspace. This limitation potentially allows for the existence of units that are not equivalent to standard exponential units, and the associated product system may fail to be completely spatial. We will address these issues elsewhere. \end{remark}

\section{A continuum of non-isomorphic product systems of type \(\rm{II}\sb{\infty}\)}\label{sec:examples}

In this section, we present a construction of a continuum of non-isomorphic, type \(\rm{II}\) product systems arising from \( B \)-valued L\'evy processes, where \( B \) is an infinite-dimensional Banach space. We show that by varying the L\'evy measures over an uncountable index set, the associated product systems exhibit structural differences that preclude isomorphism.

Let \( B = \ell^2(\mathbb{N}) \) with standard orthonormal basis \( \{e_n\}_{n=1}^\infty \). For each \( \lambda = (\lambda_1, \lambda_2, \dots) \in \Lambda:= [0,1]^\mathbb{N} \), we define the measure \( \nu_\lambda \) on \( B \) as:
\[
\nu_\lambda = \sum_{n=1}^\infty \alpha_n \left[ \lambda_n \delta_{e_n} + (1 - \lambda_n) \delta_{-e_n} \right],
\]
where \( \delta_{e_n} \) and  \( \delta_{-e_n} \) are the Dirac measures concentrated at the points \( e_n \) and \(- e_n \) in the Banach space \(B\), and \( \{\alpha_n\}_{n=1}^\infty \) is a sequence of positive weights satisfying \( \sum_{n=1}^\infty \alpha_n = 1 \) and \( \sum_{n=1}^\infty \alpha_n^2 < \infty \). We choose \( \alpha_n = 2^{-n} \) for convergence. Since \(\nu_\lambda(\{0\})=0\) and 
\[
\int _B\min(1, \|x\|^2)\nu_\lambda(dx)=\sum_{n=1}^\infty\alpha_n[\lambda_n\cdot 1+(1-\lambda_n)\cdot 1]=1,
\]
we conclude that \(\nu_\lambda\) is a L\'evy measure on \(B\).

For each \( \lambda \in \Lambda \), define the \( B \)-valued stochastic process \(\bm{X}^\lambda= \{X_t^\lambda\}_{t \geq 0} \) by
\[
X_t^\lambda = \sum_{n=1}^\infty \left[ N_n^+(t) - N_n^-(t) \right] e_n,
\]
where \( N_n^+(t) \) is a Poisson process with rate \( \mu_{1,n}=\alpha_n \lambda_n \) and \( N_n^-(t) \) is a Poisson process with rate \(\mu_{2,n}=\alpha_n (1 - \lambda_n) \).
The processes \( \{N_n^+(t), N_n^-(t)\}_{n=1}^\infty \) are independent across \(n\). 

It is clear that \(\bm{X}^\lambda\) is a L\'evy process. For example, to check stochastic continuity, we observe that the intensity of jumps is finite. Therefore, the probability of a jump in an infinitesimal interval is proportional to the interval length. As \(t\to s\), the probability of a jump exceeding any fixed \(\varepsilon >0\) tends to zero, ensuring that \(
\lim_{t\to s}\mathbb{P}(\|X_t^\lambda-X_s^\lambda\|>\varepsilon)=0.
\) 

We also observe that the L\'evy-It\^o triplet of  \(\{X_t^\lambda\}_{t \geq 0}\) is \[\left(\sum_{n=1}^\infty \alpha_n(2\lambda_n-1)e_n, 0, \nu_\lambda\right).\]
To determine the drift, we compute the expectation of \(X_t^\lambda\): 
\begin{align*}\mathbb{E}[X_t^\lambda]&=\sum_{n=1}^\infty \mathbb{E}\left[ N_n^+(t) - N_n^-(t) \right] e_n= \sum_{n=1}^\infty(\mu_{1,n}t-\mu_{2,n}t)e_n\\&= \sum_{n=1}^\infty\alpha_n(2\lambda_n-1)te_n,\end{align*}
so the drift vector is \(
b= \sum_{n=1}^\infty\alpha_n(2\lambda_n-1)e_n.\)

Next, we use the L\'evy-Khintchine formula to show that the L\'evy exponent  \( \Psi_{\bm{X}^\lambda}\) of \(\bm{X}^\lambda\) is given by \begin{align}\label{LK}\Psi_{\bm{X}^\lambda}(\phi)=\sum_{n=1}^\infty\alpha_n[(\cos \phi_n-1)+i(2\lambda_n-1)\sin \phi_n],\end{align} for all \(\phi\in B^*\), where \(\phi_n=  \langle  \phi, e_n \rangle\). Indeed, the drift term is\[ i \langle  \phi, b \rangle=i\langle  \phi, \sum_{n=1}^\infty\alpha_n(2\lambda_n-1)e_n \rangle=i\sum_{n=1}^\infty\phi_n\alpha_n(2\lambda_n-1),\] and the L\'evy measure integral is \begin{align*}
I&=\int_{B \setminus \{0\}} \left(e^{i \langle \phi, x \rangle} - 1 - i \langle \phi, x \rangle \mathbf{1}_{\|x\| \leq 1}\right) \nu_\lambda(dx)\\&=
\sum_{n=1}^\infty \alpha_n\left[\lambda_n(e^{i\phi_n}-1-i\phi_n)+(1-\lambda_n)(e^{i\phi_n}-1+i\phi_n) \right]\\&=
\sum_{n=1}^\infty \alpha_n\left[(\cos \phi_n-1)+i(2\lambda_n-1)(\sin \phi_n-\phi_n) \right].
\end{align*} Therefore \(\Psi_{\bm{X}^\lambda}(\phi)=i \langle  \phi, b \rangle+I=\sum_{n=1}^\infty\alpha_n[(\cos \phi_n-1)+i(2\lambda_n-1)\sin \phi_n\)], as claimed. 

Next, we consider the product system \( E^\lambda = \{E^\lambda(t)\}_{t > 0} \) associated with \(\bm{X}^\lambda\):
\[
E^\lambda(t) = L^2(\Omega, \mathcal{F}_t, \mathbb{P}^\lambda),
\]
where \( \Omega \) is the sample space, \( \mathcal{F}_t \) is the sigma-algebra up to time \( t \), and \( \mathbb{P}^\lambda \) is the probability measure induced by \( X_t^\lambda \). We then have the following theorem.
\begin{theorem}\label{type}
\(E^\lambda\) is a product system of type \(\rm{II}\) and infinite Arveson index, for each \(\lambda\in \Lambda\).
\end{theorem}
\begin{proof}

 For each \(\lambda\in \Lambda \) and \(n\in\mathbb{N}\), define the process \[ 
 X_{n,t}^\lambda = N_n^+(t) - N_n^-(t),\quad t\geq 0. 
 \]  Then \(\bm{X}_n^\lambda=\{X_{n,t}^\lambda\}_{t\geq 0}\) is a \(\mathbb{Z}\)-valued L\'evy process, as it has independent and stationary increments, and each random variable \(X_{n,t}^\lambda\) follows a Skellam distribution with parameters \(\mu_{1,n}t\) and \(\mu_{2,n}t\) \cite{Skellam}. 
 
 Let  \(E^\lambda_n = \{E^\lambda_n(t)\}_{t > 0}\) be the product system associated with  \(\bm{X}_n^\lambda\):\[E_n^\lambda(t)=L^2(\Omega_n, \m{F}_{n,t}, \mathbb{P}_n^\lambda),\] where \(\Omega_n\) is the sample space associated with the \(n\)-th coordinate process \(X_{n,t}^\lambda\), \(\m{F}_{n,t}\) is the \(\sigma\)-algebra up to time \(t\), and \(\mathbb{P}_n^\lambda\) is the probability measure induced by \(X_{n,t}^\lambda\).  Since the Skellam distribution has infinite support, the Hilbert space \(E_n^\lambda(t)\) is infinite dimensional.

We observe that every unit \( u_n^\lambda=\{u_n^\lambda(t)\}_{t>0}\) of the product system \(E^\lambda_n\) must be exponential unit of the form \begin{eqnarray}\label{unitP}u_n^\lambda(t)=\exp(if_nX_{n,t}^\lambda),\end{eqnarray}
 for some \(f_n\in \mathbb{R}\), because units in \(E^\lambda_n\) correspond to one-dimensional representations of the additive group generated by \(X_{n,t}^\lambda\), and the only continuous one-dimensional representations of \(\mathbb{Z}\) are characters of the form  \(x\mapsto e^{if_nx}\), for some \(f_n\in \mathbb{R}\).

 Next, since  \( X_{n,t}^\lambda\) are independent, each Hilbert space \(E^\lambda(t)\) can be realized as:  \[ E^\lambda(t) = \bigotimes_{n=1}^\infty E_n^\lambda(t):=\limind\left( \bigotimes_{n=1}^k E_n^\lambda(t) \right), \] where  the infinite tensor product of Hilbert space is realized with respect to the reference vectors \(\{\bm{1}_n^\lambda(t)\}_{n=1}^\infty\), where \(\bm{1}_n^\lambda(t)\in E_n^\lambda(t) \) is the constant function 1. In other words, if \(V_{k,k+1}^\lambda(t):\bigotimes_{n=1}^k E_n^\lambda(t)\to\bigotimes_{n=1}^{k+1} E_n^\lambda(t)\) is the isometry \(V_{k,k+1}^\lambda(t)(\phi_1\otimes \phi_2\otimes\cdots \otimes \phi_k)=\phi_1\otimes \phi_2\otimes\cdots \otimes \phi_k\otimes \bm{1}_{k+1}^\lambda(t)\) then \(\left(\bigotimes_{n=1}^{k} E_n^\lambda(t), V_{k,k+1}^\lambda(t)\right)\) is an inductive limit of Hilbert spaces and \(\bigcup_k  \bigotimes_{n=1}^{k} E_n^\lambda(t)=L^2(\times_{n=1}^k\Omega_n, \otimes_{k=1}^n\m{F}_{n,t}, \times_{k=1}^n\mathbb{P}_n^\lambda)\) is dense in  \(E^\lambda(t)\) by the martingale convergence theorem in \(L^2\). Arguing as in \cite[Theorem 4.2]{Flor}), we deduce that for every unit \(u^\lambda\) of the product system \(E^\lambda\), there exists a sequence \(\{u_n^\lambda\}_{n\in\mathbb{N}}\)  of units \(u_n^\lambda\) of the product systems \(E^\lambda_n\) such that \[ u^\lambda(t)=\bigotimes_{n=1}^\infty u_n^\lambda(t),\quad t>0.\]
 
 We then infer  from (\ref{unitP}) that units in \( E^\lambda(t) \) are constructed from exponential vectors dependent on finite linear functionals \( f \in B^* \). Specifically, units are of the form:
\[u_f^\lambda(t) = \exp\left( i \sum_{n=1}^\infty f_n X_{n,t}^\lambda \right),\]
 where \( f = (f_n)\in B^* \) has only finitely many non-zero components.

However, for any \(t>0\), consider a function \( \psi \in E^\lambda(t) \) defined by \[
 \psi = \bigotimes_{n=1}^\infty \phi_n,\] where each \( \phi_n \) is chosen such that \[
\int \phi_n e^{i f_n X_{n,t}^\lambda} \, d\mathbb{P}_n^\lambda = 0 \quad \forall f_n \in \mathbb{C}.
\] For example, \(\phi_n\) can be the centered indicator function of a non-trivial event in \(E^\lambda(t)\), e.g., \[\phi_n(x) =\left\{\begin{array}{ll}1-p_n, & x=0\\-p_n,&x\neq 0\end{array}\right.\] where \(p_n=\mathbb{P}_n^\lambda(X_{n,t}^\lambda=0).\) Then for any unit \(u_f^\lambda\) of \(E^\lambda\),  where \( f = (f_n)\in B^* \) has only finitely many non-zero components, we have \[  \langle \psi, u_f^\lambda(t) \rangle=\prod_{n=1}^\infty
\int \phi_n e^{i f_n X_{n,t}^\lambda} \, d\mathbb{P}_n^\lambda =0.\]
This implies that \( \psi \) is orthogonal to all finite linear combinations of units. Therefore, there exists \( \psi \in E^\lambda(t) \) orthogonal to the span of units, showing that units do not generate \( E^\lambda(t) \). Finally, we notice that the dimension of the space of units is infinite, as there are countably many finite linear functionals \(
f\) with finite support. The theorem is proved.
\end{proof}
\begin{theorem}\label{last}
The product systems  \( E^\lambda \) and \( E^{\lambda'} \) are not isomorphic for \( \lambda \neq \lambda' \).
\end{theorem}
\begin{proof}
We claim  that if the product systems  \( E^\lambda \) and \( E^{\lambda'} \) are isomorphic, then \( \lambda = \lambda' \). For this purpose, consider  an isomorphism of product systems \(\theta=\{\theta_t\}_{t>0}: E^\lambda \to E^{\lambda'} \). Thus, each \(\theta_t:E^\lambda(t)\to E^{\lambda'}(t)\) is a unitary operator, and they satisfy the compatibility relation \(U_{s,t}^{\lambda'}(\theta_s\otimes \theta_t)=\theta_{s+t}U_{s,t}^\lambda\) for all \(s,\,t>0\), where \(U_{s,t}^\lambda\) and \(U_{s,t}^{\lambda'}\) are the multiplications of these product systems. By normalizing if necessary, we can assume without loss of generality that \(\theta_t\bm{1}^\lambda(t) =\bm{1}^{\lambda'}(t)\), for all \(t>0\), where \(\bm{1}^\lambda(t)\in E^\lambda(t) \) and  \(\bm{1}^{\lambda'}(t)\in E^{\lambda'}(t) \) denote the constant function 1

Since isomorphisms of product systems preserve their units, we deduce from Theorem \ref{type} than for any \( f = \{f_n\}_{n\geq 1}\in B^*= \ell^2(\mathbb{N}),\) with finite support, there exists \(\hat{f} = \{\hat{f}\}_{n\geq 1}\in B^*= \ell^2(\mathbb{N})$ with finite support such that \(\theta_tu_f^\lambda(t)=u_{\hat{f}}^{\lambda'}(t)\), for all \(t>0\). Thus  \[\langle u_f^\lambda(t), u_g^\lambda(t) \rangle_{E^\lambda(t)}=\langle u_{\hat{f}}^{\lambda'}(t), u_{\hat{g}}^{\lambda'}(t) \rangle_{E^{\lambda'}(t)},\quad t>0\] for all \(f,\,g\in B^*\) having finite support. But 
\begin{align*}
\langle u_f^\lambda(t), u_g^\lambda(t) \rangle_{E^\lambda(t)}&=\mathbb{E}^\lambda\left[\exp(i\sum_{n=1}^\infty(f_n-g_n)X_{n, t}^\lambda )\right]=\mathbb{E}^\lambda\left[\exp(i\langle f-g, X_t^\lambda \rangle \right]\\&=\exp (t\Psi_{{\bm{X}}^\lambda} (f-g)),\end{align*} where 
\(\Psi_{\bm{X}^\lambda}\) is the L\'evy exponent of the process \(\bm{X}^\lambda\). 

Similarly, \(\langle u_{\hat{f}}^{\lambda'}(t), u_{\hat{g}}^{\lambda'}(t) \rangle_{E^{\lambda'}(t)}=\exp (t\psi_{{\bm{X}}^{\lambda'}} (\hat{f}-\hat{g}))\) so \(\psi_{\bm{X}^\lambda} (f-g)=\psi_{\bm{X}^{\lambda'}}(\hat{f}-\hat{g}).\) By taking \(g=0\), we obtain  \[
\psi_{\bm{X}^\lambda} (f)=\psi_{\bm{X}^{\lambda'}}(\hat{f}),
\]
for all \(f\in B^*\) with finite support. Thus, using (\ref{LK}), we have:
\[\sum_{n=1}^\infty\alpha_n[(\cos f_n-1)+i(2\lambda_n-1)\sin f_n]=\sum_{n=1}^\infty\alpha_n[(\cos \hat{f}_n-1)+i(2\lambda_n'-1)\sin \hat{f}_n].\]
Consequently, \(f_n-\hat{f}_n=2k\pi\) and \(\lambda_n=\lambda'_n\), for all \(n\in\mathbb{N}\), so \(\lambda =\lambda'\). The theorem has been proved.
\end{proof}


\begin{bibdiv}
\begin{biblist}


\bib{Applebaum}{book}{
   author={Applebaum, David},
   title={L\'{e}vy processes and stochastic calculus},
   series={Cambridge Studies in Advanced Mathematics},
   volume={116},
   edition={2},
   publisher={Cambridge University Press, Cambridge},
   date={2009},
   pages={xxx+460},
}

\bib{A2}{article}{
   author={Arveson, William},
   title={Continuous analogues of Fock space},
   journal={Mem. Amer. Math. Soc.},
   volume={80},
   date={1989},
   number={409},
   pages={iv+66},
}





\bib{A7}{book}{
   author={Arveson, William},
   title={Noncommutative dynamics and $E$-semigroups},
   series={Springer Monographs in Mathematics},
   publisher={Springer-Verlag, New York},
   date={2003},
   pages={x+434},
}




\bib{Bog1}{book}{
   author={Bogachev, Vladimir I.},
   title={Gaussian measures},
   series={Mathematical Surveys and Monographs},
   volume={62},
   publisher={American Mathematical Society, Providence, RI},
  date={1998},
   pages={xii+433},
   isbn={0-8218-1054-5},
   review={\MR{1642391}},
   doi={10.1090/surv/062},
}





\bib{Flor}{article}{
   author={Floricel, Remus},
   title={Infinite tensor products of spatial product systems},
   journal={Infin. Dimens. Anal. Quantum Probab. Relat. Top.},
   volume={11},
   date={2008},
   number={3},
   pages={447--465},
}
	






\bib{Liebscher}{article}{
   author={Liebscher, Volkmar},
   title={Random sets and invariants for (type II) continuous tensor product
   systems of Hilbert spaces},
   journal={Mem. Amer. Math. Soc.},
   volume={199},
   date={2009},
   number={930},
   pages={xiv+101},
}



\bib{Mar0}{book}{
   author={Markiewicz, Daniel },
   title={Completely positive semigroups and their product systems},
   note={Thesis (Ph.D.)--University of California, Berkeley},
   date={2002},
   pages={73},
}
	

\bib{Mar1}{article}{
   author={Markiewicz, Daniel},
   title={On the product system of a completely positive semigroup},
   journal={J. Funct. Anal.},
   volume={200},
   date={2003},
   number={1},
   pages={237--280},
}




\bib{Mey}{article}{
   author={Meyer, P.-A.},
   title={Les syst\`emes-produits et l'espace de Fock, d'apr\`es W. Arveson},
   language={French},
   conference={
      title={S\'{e}minaire de Probabilit\'{e}s, XXVII},
   },
   book={
      series={Lecture Notes in Math.},
      volume={1557},
      publisher={Springer, Berlin},
   },
   date={1993},
   pages={106--113},
}

\bib{Pow88}{article}{
   author={Powers, Robert T.},
   title={An index theory for semigroups of $^*$-endomorphisms of ${\scr
   B}({\scr H})$ and type ${\rm II}_1$ factors},
   journal={Canad. J. Math.},
   volume={40},
   date={1988},
   number={1},
   pages={86--114},
}



	
\bib{Sato1}{book}{
   author={Sato, Ken-iti},
   title={L\'{e}vy processes and infinitely divisible distributions},
   series={Cambridge Studies in Advanced Mathematics},
   volume={68},
   note={Translated from the 1990 Japanese original;
   Revised by the author},
   publisher={Cambridge University Press, Cambridge},
   date={1999},
   pages={xii+486},
}


\bib{Skeide}{article}{
   author={Skeide, Michael},
   title={L\'{e}vy processes and tensor product systems of Hilbert modules},
   conference={
      title={Quantum probability and infinite dimensional analysis},
   },
   book={
      series={QP--PQ: Quantum Probab. White Noise Anal.},
      volume={18},
      publisher={World Sci. Publ., Hackensack, NJ},
   },
   date={2005},
   pages={492--503},
}
	
\bib{Skellam}{article}{
   author={Skellam, J. G.},
   title={The frequency distribution of the difference between two Poisson
   variates belonging to different populations},
   journal={J. Roy. Statist. Soc. (N.S.)},
   volume={109},
   date={1946},
   pages={296},
}
	
\bib{Sundar}{article}{
   author={Shanmugasundaram, Sundar},
   title={Decomposable product systems associated to non-stationary Poisson
   processes},
   journal={Int. Math. Res. Not. IMRN},
   date={2023},
   number={13},
   pages={11432--11452},
}

\bib{Tsirelson2000}{article}{
   author={Tsirelson, Boris},
   title={From random sets to continuous tensor products: answers to three questions of W. Arveson},
   journal={arXiv:math/0001070 },
   date={2000},
}


\bib{Tsirelson2002}{article}{
   author={Tsirelson, Boris},
   title={Non-isomorphic product systems},
   conference={
      title={Advances in quantum dynamics},
      address={South Hadley, MA},
      date={2002},
   },
   book={
      series={Contemp. Math.},
      volume={335},
      publisher={Amer. Math. Soc., Providence, RI},
   },
   date={2003},
   pages={273--328},
}
	
\bib{Tsirelson2004}{article}{
   author={Tsirelson, Boris},
   title={Nonclassical stochastic flows and continuous products},
   journal={Probab. Surv.},
   volume={1},
   date={2004},
   pages={173--298},
}




\end{biblist}
\end{bibdiv}
\end{document}